\documentclass[a4paper]{amsart}

\usepackage{amsrefs}

\usepackage{etex}
\usepackage{hyperref}

\usepackage{txfonts,amsmath,amstext,amsthm,amscd,amsopn,verbatim,amssymb,amsfonts,mathtools,enumitem}
\usepackage{fullpage}

\usepackage{todonotes}

\usepackage{adjustbox}

\usepackage[bbgreekl]{mathbbol}

\usepackage{tikz}
\usepackage{tikz-cd}
\usetikzlibrary{matrix}
\usetikzlibrary{shapes}
\usetikzlibrary{arrows, automata}
\usetikzlibrary{calc,3d}
\usetikzlibrary{decorations,decorations.pathmorphing,decorations.pathreplacing,decorations.markings}
\usetikzlibrary{through}

\tikzset{snakeit/.style={decorate, decoration={snake, amplitude=.2mm,segment length=1mm}}}

\tikzset{ext/.style={circle, draw,inner sep=1pt}, int/.style={circle,draw,fill,inner sep=1pt},nil/.style={inner sep=1pt}}
\tikzset{cy/.style={circle,draw,fill,inner sep=2pt},scy/.style={circle,draw,inner sep=2pt},scyx/.style={draw,cross out,inner sep=2pt},scyt/.style={draw,regular polygon,regular polygon sides=3,inner sep=0.95pt}}
\tikzset{exte/.style={circle, draw,inner sep=3pt},inte/.style={circle,draw,fill,inner sep=3pt}}
\tikzset{diagram/.style={matrix of math nodes, row sep=3em, column sep=2.5em, text height=1.5ex, text depth=0.25ex}}
\tikzset{diagram2/.style={matrix of math nodes, row sep=0.5em, column sep=0.5em, text height=1.5ex, text depth=0.25ex}}
\tikzset{rowcolsep/.style={column sep=.2cm, row sep=.1cm}}

\tikzset{
  crossed/.style={
    decoration={markings,mark=at position .5 with {\arrow{|}}},
    postaction={decorate},
    shorten >=0.4pt}}

\tikzset{every picture/.style={baseline=-.65ex} }
\tikzset{every loop/.style={draw}}

  \tikzset{->-/.style={decoration={
    markings,
    mark=at position .5 with {\arrow{>}}},postaction={decorate}}}

\theoremstyle{plain}

\newtheorem{thm}{Theorem}[section]

\newtheorem{prop}[thm]{Proposition}
\newtheorem{cor}[thm]{Corollary}
\newtheorem{lemma}[thm]{Lemma}

\theoremstyle{definition}

\newtheorem{rem}[thm]{Remark}

\hypersetup{pdfauthor={Thomas Willwacher}}

\newcommand{\alg}[1]{\mathfrak{{#1}}}

\newcommand{\C}{{\mathbb{C}}}

\newcommand{\Z}{{\mathbb{Z}}}

\newcommand{\Q}{{\mathbb{Q}}}

\newcommand{\Gpd}{{\mathcal{G}\mathrm{pd}}}

\newcommand{\BV}{\mathsf{BV}}

\newcommand{\trunc}{\mathrm{trunc}}

\newcommand{\Com}{\mathsf{Com}}

\newcommand{\G}{\mathrm{G}}

\newcommand{\Exp}{\mathrm{Exp}}

\newcommand{\bpm}{\begin{pmatrix}}
\newcommand{\epm}{\end{pmatrix}}

\newcommand{\MC}{\mathsf{MC}}

\newcommand{\hotimes}{\mathbin{\hat\otimes}}

\newcommand{\grt}{\alg {grt}}

\newcommand{\e}{\mathsf{e}}
\newcommand{\Op}{\mathcal{O}\mathrm{p}}
\newcommand{\WOp}{\mathcal{WO}\mathrm{p}}
\newcommand{\Mon}{\mathcal{M}\mathrm{on}}

\newcommand{\La}{\Lambda}

\newcommand{\HOpc}{\mathrm{HOp}^c}
\newcommand{\cat}{\mathsf}
\newcommand{\Free}{\mathbb{F}}

\newcommand{\Seq}{\mathcal{S}\mathit{eq}}

\newcommand{\lD}{\mathsf{D}}
\newcommand{\flD}{\mathsf{fD}}

\newcommand{\Map}{\mathrm{Map}}
\newcommand{\Mor}{\mathrm{Mor}}
\newcommand{\Aut}{\mathrm{Aut}}
\newcommand{\pAut}{\widetilde{Aut}}

\newcommand{\gr}{\mathrm{gr}}

\newcommand{\GRT}{\mathrm{GRT}}
\newcommand{\GT}{\mathrm{GT}}
\newcommand{\fc}{{\mathfrak c}}

\newcommand{\beq}[1]{\begin{equation}\label{#1}}
\newcommand{\eeq}{\end{equation}}

\newcommand{\dgca}{\mathrm{dgca}}

\newcommand{\GOp}{G\Op}
\newcommand{\OpG}{\Op^{G/}}
\newcommand{\OphG}{\Op^{\hat G/}}

\newcommand{\SO}{\mathrm{SO}}
\renewcommand{\O}{\mathrm{O}}

\newcommand{\ar}{\mathrm{ar}}

\newcommand{\sset}{\mathit{s}\mathcal{S}\mathit{et}}

\newcommand{\fg}{\mathfrak{g}}

\newcommand{\sSet}{\sset}

\DeclareMathAlphabet{\mathsfit}{OT1}{cmss}{m}{sl}

\DeclareMathOperator{\POp}{{\mathsfit{P}}}
\DeclareMathOperator{\QOp}{{\mathsfit{Q}}}

\newcommand{\Pair}{{\mathcal P\mathit{air}}}

\DeclareMathOperator{\id}{\mathit{id}}

\newcommand{\pab}{\mathsf{PaB}}
\newcommand{\parb}{\mathsf{PaRB}}
\newcommand{\hpab}{\widehat{\pab}}
\newcommand{\hparb}{\widehat{\parb}}
\newcommand{\B}{\mathcal{B}}

\author{Geoffroy Horel}
\address{Institut Galilée\\
Université Sorbonne Paris Nord\\
99 avenue Jean-Baptiste Clément \\
93430 Villetaneuse, France. }
\email{horel@math.univ-paris13.fr}

\author{Thomas Willwacher}
\address{Department of Mathematics \\ ETH Zurich \\
R\"amistrasse 101 \\
8092 Zurich, Switzerland}
\email{thomas.willwacher@math.ethz.ch}

\thanks{G.H. has been partially supported by the Agence Nationale pour la recherche, project number ANR-20-CE40-0016 HighAGT. T.W. has been partially supported by the NCCR Swissmap, funded by the Swiss National Science Foundation}

\begin{document}
\title{Automorphisms of framed operads}
\begin{abstract}
Let $\POp$ be an operad acted upon by a group $G$, and let $\QOp=\POp\rtimes G$ be the corresponding framed operad.
We relate the homotopy automorphism groups of $\POp$ and $\QOp$. 
We apply the result to compute the automorphisms of the framed little disks operad.
\end{abstract}

\maketitle

\section{Introduction}

Let $\POp$ be a (topological or simplicial) operad acted upon by a (topological or simplicial) group $G$.
Then the $G$-framed operad $\POp\rtimes G$ is defined such that
\[
    \POp\rtimes G(r)
    =
    \POp(r) \times G^{\times r}
\]
with the composition operations
\begin{equation}  \label{equ:semidirect comp} 
\begin{gathered}
    \circ: \POp\rtimes G(r)\times \POp\rtimes G(s_1)\times \cdots 
    \POp\rtimes G(s_r)\to \POp\rtimes G(s_1+\dots+s_r)
    \\ 
    \left(
    (p,g_1,\dots,g_r)
    \times 
    (q_1,h_{11},\dots,h_{1s_1}),
    \dots,
    (q_r,h_{r1},\dots,h_{rs_r})
    \right)
    =
   (p\circ(g_1\cdot q_1,\dots,g_r\cdot q_r),
   g_1h_{11},\dots,g_1h_{1s_1},
   g_2h_{21},\dots,g_rh_{rs_r}).
\end{gathered}
\end{equation}
A $\POp\rtimes G$-algebra in spaces is the data of a $G$-space $X$ equipped with a $\POp$-algebra structure that is $G$-equivariant in the sense that the structure maps
\[\POp(n)\times X^n\to X\]
are $G$-equivariant, when the source is given its diagonal $G$-action. More generally, this construction is left adjoint to the forgetful functor from operads under $G$ to operads in $G$-spaces (see Lemma \ref{lem:semidirect adj}).

The classical and motivating example of this construction is the little disks operads $\POp=\lD_2$ that is acted upon by the group $G=\SO(2)$, with $\lD_2\rtimes \SO(2)=: \flD_2$ the framed little disks operad and its higher dimensional variants $\flD_n=\lD_n\rtimes\SO(n)$.

Generally, one may hence associate four automorphism spaces to a $G$-operad $\POp$: The (homotopy) automorphism space $\Aut^h_{\Op}(\POp)$ of $\POp$ as an ordinary operad, or as a $\G$-operad $\Aut^h_{G\Op}(\POp)$, or the homotopy automorphisms $\Aut^h_{\Pair}((G,\POp))$ of the pair $(G,\POp)$ consisting of the group $G$ and the operad $\POp$, 
or one may consider the automorphisms of the framed operad $\Aut^h_{\Op}(\POp\rtimes G)$.
The purpose of the present paper is a comparison of these four automorphism spaces. To this end we have the following result.

\begin{thm}\label{thm:main}
    Let $\POp$ be (topological or simplicial) operad acted upon by a (topological or simplicial) group $G$ such that $\POp(0)=*$, $\POp(1)$ is contractible, and let $\QOp=\POp\rtimes G$ be the corresponding framed operad. Then the following holds. 
    \begin{enumerate}
        \item The (derived) morphism of simplicial monoids 
        \[
        \Aut^h_{\Pair}((G,\POp))\to \Aut^h_{\Op}(\QOp)
        \]
        induced by the functoriality of the semi-direct product construction $\rtimes$ is a weak equivalence.
    In particular, there is a homotopy fiber sequence of simplicial monoids 
    \begin{equation}\label{equ:main fib seq}
        \Aut^h_{\GOp}(\POp) \to  \Aut^h_{\Op}(\QOp) \to \Aut^h_{Mon}(G).
    \end{equation}
    \item Let $f:G\to \Aut^h_{\Op}{\POp}$ be the map defining the $G$-action on $\POp$.
Then the classifying space
$B\Aut^h_{\GOp}(\POp)$ is weakly equivalent to the connected component corresponding to $f$ of the unbased mapping space between the classifying spaces of $G$ and $\Aut^h_{\Op}(\POp)$
\[
    B\Aut^h_{G\Op}(\POp)
   \simeq \Map\left(BG, B\Aut^h_{\Op}(\POp)\right)_{Bf}.
\]
    \end{enumerate}
\end{thm}

We apply the above result to the (framed) little disks operad and its rationalization. 
The homotopy automorphisms of $\lD_2$ have been computed by the first author, building on earlier work by Drinfeld \cite{Drinfeld}.
\begin{thm}[{Horel \cite[Theorem 8.5]{HorelProfinite}}]
    \label{thm:aut lD top}
    There is a weak equivalence of simplicial monoids 
    \[
        \Aut^h_{\La\Op}(\lD_2)
        \cong 
        \O(2).
    \]
\end{thm}

The homotopy automorphism space of the Bousfield-Kan rationalization of the little disks operad has been computed by B. Fresse.
\begin{thm}[{Fresse \cite[Theorem A in part III]{Frbook} }]
    \label{thm:aut lD Q}
There is a weak equivalence of simplicial monoids 
\[
    \Aut^h_{\Op}(\lD_2^{\mathbb{Q}})
    \cong 
    \GRT \ltimes \SO(2)^{\mathbb Q},
\]
with $\GRT$ the Grothendieck-Teichmüller group.
\end{thm}

Using Theorem \ref{thm:main} above, we then obtain the following.

\begin{thm}
    \label{thm:main flD top}
    There are weak equivalences of simplicial monoids 
    \begin{align*}
        \Aut^h_{\SO(2)\Op}(\lD_2) & \cong \SO(2) \\
        \Aut^h_{\Op}(\flD_2) &\cong \O(2).
    \end{align*}
\end{thm}

\begin{thm}\label{thm:main flD Q}
    There are weak equivalences of simplicial monoids 
    \begin{align*}
        \Aut^h_{\SO(2)^{\Q}\Op}(\lD_2^\Q) & \cong \GRT_1 \ltimes \SO(2)^\Q \\
        \Aut^h_{\Op}(\flD_2^{\Q}) &
        \cong 
        \GRT \ltimes \SO(2)^{\mathbb Q}.
    \end{align*}
    with $\GRT_1$ the pro-unipotent Grothendieck-Teichmüller group and $\GRT=\Q^\times \ltimes \GRT_1$, see \cite{BarNatan,Drinfeld}.
\end{thm}

There is also a version of this theorem for the profinite completion of $\flD_2$ recovering the main result of \cite{BHR} and giving a computation of the fundamental group of the group of homotopy automorphisms of $\widehat{\flD}_2$ which was not done in \cite{BHR}. In order to phrase this result we have to address the technicality that the profinite completion functor does not preserve products in general. The way to deal with this, introduced in \cite{HorelProfinite}, is to use the category of weak operads (denoted $\WOp$). This category is the category of functors from the algebraic theory of operads that preserve products up to homotopy. Our theorem in this setting gives the following.

\begin{thm}
    \label{thm:main flD profinite}
    There are weak equivalences of simplicial monoids 
    \begin{align*}
        \Aut^h_{|\widehat{\SO(2)}|\WOp}(\widehat{\lD}_2) & \cong \widehat{\GT}_1 \ltimes \widehat{\SO(2)} \\
        \Aut^h_{\WOp}(\widehat{\flD}_2) &
        \cong 
        \widehat{\GT} \ltimes \widehat{\SO(2)}.
    \end{align*}
     with $\widehat{\GT}$ the pro-finite Grothendieck-Teichmüller group and $\widehat{\GT}_1$, the kernel of the cyclotomic character $\widehat{\GT}\to\widehat{\Z}^\times$. See \cite{Drinfeld}.
\end{thm}

We emphasize that the version of the little disks operad that we use has an operation in arity zero, i.e., $\lD(0)=*$. Composition with  this element is the same as forgetting disks from a configuration. However, there are also similar results for the non-unital version (see Theorems \ref{thm:non-unital thm} and \ref{thm:non-unital thm profinite}).

\section{Model categories, functors and adjunctions}

\subsection{Model category structures}
Fix a simplicial group $G$.
We consider the following categories: 
\begin{itemize}
\item The category $\sSet$ of simplicial sets and $\Seq$ of symmetric sequences in simplicial sets.
\item The category of simplicial operads $\Op$. Our operads may have nullary operations. 
\item The category $\Mon$ of monoids in simplicial sets.
This can also be understood as the subcategory $\Mon\subset \Op$ of operads with only unary operations, by considering a monoid $M$ as an operad such that 
\begin{equation}\label{equ:Mon Op incl}
M(r) 
=
\begin{cases}
    M & \text{for $r=1$} \\
    \emptyset & \text{otherwise}
\end{cases} .
\end{equation}
This also allows us to consider the under-category $\OpG$.
\item The category of operads with a $G$-action $\GOp$.
\item The category $\Pair$ of pairs $(G,\POp)$ consisting of a simplicial group $G$ and a $G$-operad $\POp$. The morphisms $(G,\POp)\to (H,\QOp)$ are pairs $(\phi, F)$ consisting of a morphism of simplicial groups $\phi:G\to H$ and a morphism of $G$-operads $F:\POp\to \phi^*\QOp$.
\end{itemize}
    
We equip $\sSet$ with the standard Quillen model structure, and the other categories above with cofibrantly generated model structures by transfer along the forgetful functors 
\begin{align*}
    &\GOp \to \Op \to \Seq \to \prod_{r\geq 0} \sSet
    &
    &\Pair\to \sSet \times \prod_{r\geq 0} \sSet.
\end{align*}
Concretely, this means that in each case the weak equivalences (resp. fibrations) are arity and/or objectwise weak equivalences (resp. fibrations) of simplicial sets. The cofibrations are those morphisms that have the left-lifting property with respect to acyclic fibrations.
The generating (acyclic) cofibrations are the images of the generating cofibrations in $\sSet$ under the respective free object functors.

\begin{prop}[Berger-Moerdijk]
The above classes of distinguished morphisms define cofibrantly generated model category structures on the categories $\Op$, $\Mon$, $G\Op$, $\Pair$.
\end{prop}
\begin{proof}
For the case of $\Op$ this is \cite[Theorem 3.2]{BMAxiomatic} (see also section 3.3.1 in that paper). For the other cases the proof is identical, one just replaces "operad" by monoid, $G$-operad or pair.
Alternatively, the proposition is a special case of \cite[Theorem 2.1]{BMColored}, since the above types of algebraic objects are all algebras over suitable colored operads.
\end{proof}

We call the resulting model category structures the projective model category structures. 
The under-category $\OpG$ can then simply be equipped with the slice model structure. This means that a morphism is a weak equivalence (resp. fibration, cofibration) iff it is a weak equivalence (resp. fibration, cofibration) in the underlying category $\Op$.

\subsection{Functors and adjunctions}
The semidirect product functor 
\[
\rtimes  \colon \Pair \to \Op     
\]
associates to a pair $(G,\POp)$ of a simplicial group $G$ and an operad $\POp$ with a $G$-action the operad $\POp\rtimes G$ such that
$$
(\POp\rtimes G)(r)= \POp(r)\times G^{\times r}.
$$
The compositions are defined via \eqref{equ:semidirect comp}.
The operad $\POp\rtimes G$ comes with a natural action of $G$, and a natural map $G\to \POp\rtimes G$.


\begin{lemma}\label{lem:semidirect adj}
We have a Quillen adjunction 
\[
  (-)\rtimes G \colon G\Op
   \rightleftarrows \OpG \colon \iota,
\]
with $\iota$ the forgetful map from operads under $G$ to operads with a $G$-action.
\end{lemma}
\begin{proof}
To check the adjunction relation 
\[
\Mor_{G\Op}(\POp, \iota \QOp) \cong 
\Mor_{\OpG}(\POp\rtimes G, \QOp)
\]
note that $\POp\rtimes G$ is generated by $\POp$ and $G$, with relations those in $\POp$ and $G$ and additionally the relations
\[
g\circ p \circ (g^{-1},\dots,g^{-1})
=
g\cdot p.    
\]
This implies that an operad map from $\POp\rtimes G$ that is fixed on $G$ is the same as a map from $\POp$ that respects the $G$-action.
It is also clear that $\iota$ preserves weak equivalences and fibrations, since they are created in $\Seq$, and is hence right Quillen.
\end{proof}

\begin{lemma}\label{lem:one adj}
We have a Quillen adjunction 
\[
 i \colon \Mon \rightleftarrows \Op \colon (-)(1),    
\]
where the left adjoint $i$ is the inclusion of monoids into operads, see \eqref{equ:Mon Op incl}, and the right adjoint associates to the operad $\POp$ the monoid $\POp(1)$.

The adjunction counit 
\[
\POp(1) \to \POp
\]
is a cofibration in $\Op$ for any cofibrant operad $\POp$.
\end{lemma}
The Lemma seems to be known to experts, but we failed to find a citeable reference.
\begin{proof}[Proof sketch.]
The adjunction relation is again (fairly) obvious and left to the reader.
It is clearly a Quillen adjunction since the right adjoint preserves weak equivalences and fibrations, which are such morphisms arity-wise on the level of simplicial sets.

For the last assertion let $\mathcal X$ be the class of all operads $\POp$ such that the adjunction counit $\POp(1)\to \POp$ is a cofibration. Then one checks that $\mathcal X$ is closed under retracts and filtered colimits. It also contains all free objects, in particular domains and targets of the generating (acyclic) cofibrations. One also checks that $\mathcal X$ is closed under pushouts along cofibrations between objects in $\mathcal X$.

But by \cite[Proposition 2.1.18]{Hovey} any cofibrant object in a cofibrantly generated model category is a retract of a cell complex, i.e., a colimit along a transfinite composition $* \to \cdots \to X_n \to X_{n+1}\to\cdots$ of morphisms that are each pushouts along generating cofibrations.
Hence, by transfinite induction, we have that each object $X_n$ and thus the colimit is in $\mathcal X$. 
\end{proof}

\subsection{Lambda operads and Reedy model structure}
 
Let $\Op_*\subset \Op$ be the full subcategory of operads $\POp$ such that $\POp(*)=*$.
Let $\La$ be the category with objects the non-negative integers, and morphisms $m\to n$ the injective (not necessarily order preserving) maps 
\[
\{1,\dots,m\} \to \{1,\dots,n\}.
\]
We have a forgetful functor 
\[
F\colon \Op_* \to \La\sSet := \sSet^{\La^{op}},
\]
that sends an operad $\POp$ to the (positive arity part of the) underlying symmetric sequence, equipped with the operations of operadic composition with $*$.
The category $\La$ is a generalized Reedy category and hence $\La\sSet$ is equipped with the Reedy model structure, see \cite[Theorem 8.3.19]{Frbook}.

Following Fresse \cite{Frextended,Frbook} we define the Reedy model structure on $\Op_*$ to be the one obtained by right transfer along $F$ from the Reedy model structure on $\La\sSet$.

\begin{lemma}\label{lem:Op star adj}
There is a Quillen adjunction with respect to the Reedy model structure on $\Op_*$
\[
   (-)(1) \colon \Op_* \rightleftarrows \Mon \colon \Com \rtimes (-)
\] 
with the left-adjoint being the forgetful functor that takes the unary part $\POp(1)$ of an operad $\POp$.
\end{lemma}
\begin{proof}
One easily verifies the adjunction relation.

To check that the adjunction is Quillen, note that we have a commutative diagram 
\[
\begin{tikzcd} 
    \Op_* \ar{d} & \Mon \ar{l}{\Com \rtimes (-)} \ar{d} \\
    \La\sSet & \sSet\ar{l}{\Free^c_\La}
\end{tikzcd}  
\]
with $\Free^c_\La$ the cofree $\La$ object, defined such that 
\[
    \Free^c_\La(X)(r) = X^{\times r}.
\]
All arrows in the diagram are right adjoints. The model structures of the categories in the upper row are defined by transfer along the vertical forgetful functors. It follows that $\Com \rtimes (-)$ is right Quillen if $\Free^c_\La$ is right Quillen. But this follows if its left adjoint (forgetful) functor $\La\sSet\to \sSet$ is left Quillen. But by \cite[Theorem II.8.3.20]{Frbook} the (acyclic) cofibrations in $\La\sSet$ are the morphisms that are (acyclic) cofibrations in $\Seq$.
In particular the arity one part of a symmetric sequence is just a simplicial set, and hence the forgetful functor does indeed preserve (acyclic) cofibrations. 
\end{proof}

We will need below the following corollary.

\begin{cor}\label{cor:only one}
    Let $\POp\in \Op_*$ be an operad such that $\POp(1)\simeq *$ and let $G$ be a simplicial monoid. Then we have that
    \begin{equation}\label{equ:prop only one}
    \Map_{\Op}^h(\POp, \Com\rtimes G) \simeq *.
    \end{equation}
    \end{cor}
    \begin{proof}
        By  \cite[Theorem 1]{FTW3} the inclusion $\Op_*\subset \Op$ is homotopically fully faithful, so that
        \[
            \Map_{\Op}^h(\POp, \Com\rtimes G) \simeq 
            \Map_{\Op_*}^h(\POp, \Com\rtimes G).
        \] 

The Corollary then follows immediately from the Quillen adjunction of Lemma \ref{lem:Op star adj}. 
    

    \end{proof}

\section{Proof of Theorem \ref{thm:main}}

\subsection{Two fibration lemmas}

\begin{lemma}\label{lem:unary fib}
    Let $\POp,\QOp\in \Op$ be operads such that $\POp$ is cofibrant and $\QOp$ is fibrant. Then the restriction map to the unary part 
    \begin{equation}\label{equ:unary fib}
    \Map_{\Op}(\POp, \QOp) \to \Map_{Mon}(\POp(1), \QOp(1))    
    \end{equation}
    is a fibration. The fiber over a morphism $f:\POp(1)\to \QOp(1)$ is $\Map_{\Op^{\POp(1)/}}(\POp,\QOp)$, where $\QOp$ is made an operad under $\POp(1)$ via $f$.
\end{lemma}
\begin{proof}
The inclusion $\POp(1)\to \POp$ is a cofibration by Lemma \ref{lem:one adj}.
By the adjunction of that Lemma we also have that 
\[
    \Map_{\Mon}(\POp(1), \QOp(1)) 
    = 
    \Map_{\Op}(\POp(1), \QOp).  
\]
Furthermore, the morphism \eqref{equ:unary fib} of the Proposition is obtained by precomposition with the map $\POp(1)\to \POp$. But because that map is a cofibration we know that \eqref{equ:unary fib} is a fibration. The fiber over the identity is evidently $\Map_{\Op^{\POp(1)/}}(\POp,\QOp)$.
\end{proof}

\begin{lemma} \label{lem:QG ComG fib} 
        Let $\POp,\QOp\in G\Op$ be $G$-operads such that $\POp$ is cofibrant and $\QOp$ is fibrant. Also suppose $G$ is fibrant.
        Then the map 
        \begin{equation}\label{equ:framed fib}
            \Map_{\GOp}(\POp, \QOp\rtimes G) 
            \to
            \Map_{\GOp}(\POp, \Com\rtimes G)    
        \end{equation}
        obtained by composition with the canonical $G$-operad morphism $\QOp\rtimes G\to \Com\rtimes G$ is a fibration with fiber 
        \[
        \Map_{\GOp}(\POp, \QOp).    
        \]
\end{lemma}
\begin{proof}
It suffices to check that the morphism $\QOp\rtimes G\to \Com\rtimes G$ is a fibration.
But by definition of the model structure this means that in each arity $r$ the morphism 
\[
    (\QOp\rtimes G)(r) = \QOp(r)\times G^{\times r} 
    \to 
    (\Com\rtimes G)(r)= G^{\times r}
\]
is an $\sSet$-fibration.
But the morphism is a product of two fibrations, namely the map $\QOp(r)\to *$ (by fibrancy of $\QOp$) and the identity on the factor $G^{\times r}$, and hence itself a fibration. Again, the identification of the fiber is obvious.
\end{proof}

\subsection{Proof of the first part of Theorem \ref{thm:main}}
\begin{lemma}\label{lem:pre part1}
The sequence \eqref{equ:main fib seq} is a homotopy fiber sequence of simplicial sets.
\end{lemma}
\begin{proof}
    Let $\hat \QOp$ be a fibrant and cofibrant replacement of $\QOp:=\POp\rtimes G$.
    Then we have that 
    \[
    \Aut^h_{\Op}(\QOp) := \Map_{\Op}'(\hat \QOp,\hat \QOp),    
    \]
    where $'$ indicates that we take the subspace consisting of the connected components of homotopy invertible morphisms.
    Applying Lemma \ref{lem:unary fib} we obtain a fibration 
    \begin{equation}\label{equ:QOp fib 1a}
        \Map_{\Op}'(\hat \QOp,\hat \QOp) \to 
        \Map_{\Mon}'(\hat \QOp(1),\hat \QOp(1)).
    \end{equation}
    Since $\hat G:=\hat \QOp(1)$ is a fibrant and cofibrant replacement for $G=\QOp(1)$, the right-hand side above is
    \[
        \Map_{\Mon}'(\hat \QOp(1),\hat \QOp(1)) =: \Aut^h_{\Mon}(G).    
    \]
    On the other hand, the fiber of \eqref{equ:QOp fib 1a} over the identity (and hence over any other point as well) is 
    \[
        \pAut_{\OphG}(\hat \QOp):= \Map'_{\OphG}(\hat \QOp, \hat \QOp ).
    \]
    Let $\hat \POp$ be a cofibrant replacement of $\POp$ in the category $\hat G\Op$. Then we have that $\hat G\to \hat \POp\rtimes\hat G$ is a cofibrant object of $\OphG$ by the Quillen adjunction of Lemma \ref{lem:semidirect adj}. It is also weakly equivalent to $\QOp$. Hence
    \[
    \Map'_{\OphG}(\hat \QOp, \hat \QOp )
    \simeq 
    \Map'_{\OphG}(\hat \POp\rtimes\hat G, \hat \QOp )
    \cong
    \Map'_{\hat G\Op}(\hat \POp, \hat \QOp ),
    \]
    where in the last step we again used the adjunction of Lemma \ref{lem:semidirect adj}.
    
    Next we apply Lemma \ref{lem:QG ComG fib} to see that $\Map'_{\hat G \Op}(\hat \POp, \hat \QOp )$ fits into a homotopy fiber sequence 
    \[
        \Map'_{\hat G \Op}(\hat \POp, \hat \POp )
        \to 
        \Map'_{\hat G \Op}(\hat \POp, \hat \QOp )
        \to 
        \Map_{\hat G \Op}(\hat \POp, \Com\rtimes G ).
    \]
    Since the base is contractible by Corollary \ref{cor:only one} we know that indeed $\Map'_{\hat G \Op}(\hat \POp, \hat \POp )
    \simeq
    \Map'_{\hat G \Op}(\hat \POp, \hat \QOp )$.
    We conclude that 
    \[
        \pAut_{\hat G\Op}(\hat \POp)
        \simeq \pAut_{\OphG}(\hat \QOp)
    \]
    as desired.
    Finally, the categories $\hat G\Op$ and $G\Op$ are Quillen equivalent (see, e.g., \cite[Theorem 16.A]{Frmodules}), and hence we have $\pAut_{\hat G \Op}(\hat \POp)=\Aut_{\hat G \Op}^h(\hat \POp)\simeq \Aut_{G \Op}^{h}(\POp)$.
\end{proof}

Now we continue with the proof of the first part of Theorem \ref{thm:main}.
We have a homotopy commutative diagram 
$$
    \begin{tikzcd}
        \Aut^h_{G\Op}(\POp)
        \ar{r}\ar{d}{=} &
        \Aut^h_{\Pair}((G,\POp)) \ar{r}\ar{d} &
        \Aut^h_{\Mon}(G) \ar{d}{=} \\
        \Aut^h_{G\Op}(\POp) \ar{r} &
        \Aut^h_{\Op}(\POp\rtimes G) \ar{r} &
        \Aut^h_{\Mon}(G).
    \end{tikzcd}
$$
The top and bottom row are homotopy fiber sequences, due to Lemma \ref{lem:pre part1}. Hence from the associated diagram of long exact sequences of homotopy groups we conclude that $\Aut^h_{\Pair}((G,\POp)) \simeq \Aut^h_{\Op}(\POp\rtimes G)$ as simplicial monoids.

\subsection{Proof of the second part of Theorem \ref{thm:main}}
The second assertion of Theorem \ref{thm:main} is a special case of the following general result on $\infty$-categories.

\begin{prop}\label{prop : equivariant Aut}
    Let $C$ be an $\infty$-category. Let $G$ be a grouplike $E_1$-space and $X$ an object $C^{BG}$. Let $f:BG\to BAut_C(X)$ be the map giving $X$ its action of $G$, then there is a weak equivalence
    \[B\Aut_{C^{BG}}(X)\simeq \Map(BG,B\Aut_{C}(X))_f\]
    where the $f$ subscript notation denotes the connected component of the mapping space containing the map $f$. 
    \end{prop}
    
    \begin{proof}
    For an $\infty$-category $C$, we denote by $C^{\simeq}$ the maximal $\infty$-groupoid contained in $C$. The functor $C\mapsto C^{\simeq}$ is right adjoint to the inclusion functor from $\infty$-groupoids to $\infty$-categories. From this universal property, we see that there is an equivalence of $\infty$-groupoids
    \[(C^{BG})^{\simeq}\simeq (C^{\simeq})^{BG}\]
    restricting this equivalence to the connected component of $X$, we get exactly the desired equivalence.
    \end{proof}

    Now every simplicial model category $M$ determines an $\infty$-category that we denote by $M_{\infty}$. The morphism spaces in the $\infty$-category are weakly equivalent to the derived mapping spaces of the model category, see \cite[Theorem 4.6.8.5]{Kerodon}. Moreover, if $G$ is a grouplike simplicial monoid, there is an equivalence of infinity-categories
    \[(\Op)_\infty^{BG}\simeq (G\Op)_\infty\]
    by \cite[Proposition 4.2.4.4]{HTT}
    We may hence apply the above proposition also to the category $\Op$ and obtain that 
    \[
    B\Aut_{G\Op}^h(\POp) \simeq 
    \Map_{\sSet}^h\left(BG, B\Aut_{\Op}^h(\POp)\right)_{Bf}
    \] 
    as desired. This concludes the proof of the second part of theorem \ref{thm:main}.
    \hfill\qed 
    
\section{Application to the framed little $n$-disks operads}

The goal of this section is to show Theorems \ref{thm:main flD top}, \ref{thm:main flD Q} and \ref{thm:main flD profinite}.

\subsection{Proof of Theorem \ref{thm:main flD top}}

Thanks to the first part of Theorem \ref{thm:main}, we have a fiber sequence
\[\Aut_{\Op^{BS^1}}(\lD_2)\to \Aut_{\Op}(\flD_2)\to\Aut_{\Mon}(\SO(2))\]
The group $\Aut^h_{\Mon}(\SO(2))$ is identified with the group of homotopy automorphisms of $B\SO(2)\simeq \mathbb{C}P^{\infty}$ in the category of based spaces. Since $\mathbb{C}P^\infty$ is a $K(\mathbb{Z},2)$, we have
\[\pi_i\Map_*(\mathbb{C}P^\infty,\mathbb{C}P^\infty)\simeq\tilde{H}^{2-i}(\mathbb{C}P^\infty;\mathbb{Z})\]
It follows that $\Aut^h_{\Mon}(\SO(2))$ is the discrete group $\mathbb{Z}/2$.

By the second part of Theorem \ref{thm:main} and Theorem \ref{thm:aut lD top}, we have 
\[B\Aut_{\SO(2)\Op}(\lD_2)\simeq \Map(BS^1,B\O(2))_f\]
with $f:\SO(2)\to \O(2)\simeq \Aut^h_{\Op}(D_2)$ the action map. This map $f$ can be identified with the inclusion $\SO(2)\to \O(2)$. Since $\SO(2)$ is connected, we have
\[\Map(B\SO(2),B\O(2))_f\simeq \Map(B\SO(2),B\SO(2))_{id}\]
and this last mapping space can be computed similarly to what we just did. In the end, we find
\[\Aut_{\SO(2)\Op}(\lD_2)\simeq \SO(2)\]
so that we have a fiber sequence
\[\SO(2)\to \Aut^h_{\cat{Op}}(\flD_2)\to \mathbb{Z}/2\]
But we have an action of $O(2)$ on $\flD_2$ that fits into a map of fiber sequences 
\[
\begin{tikzcd}
    \SO(2)\ar{r}\ar{d}{=} & O(2)\ar{r}\ar{d}  & \Z/2 \ar{d}{=}
    \\  
    \SO(2)\ar{r} & \Aut^h_{\Op}(\flD_2)\ar{r}  & \Z/2
\end{tikzcd} 
\]
from which we conclude that $\Aut^h_{\Op}(\flD_2)\simeq O(2)$ as desired.
\hfill \qed

\subsection{Dg Lie algebras and rational homotopy theory}
For $\fg$ a filtered complete dg Lie algebra we consider the Maurer-Cartan space 
\[
\MC_\bullet(\fg) = \MC(\fg\hotimes \Omega(\Delta^\bullet))    
\]
and the exponential group 
\[
\Exp_\bullet(\fg) = Z(\fg\hotimes \Omega(\Delta^\bullet))    
\]
with $Z(-)$ taking the degree zero cocycles.
It is known \cite[Theorem 5.2]{BerglundMC} that
\[
B\Exp_\bullet(\fg^\alpha) \simeq \MC_\bullet(\fg)_{\alpha}
\]
where $\alpha\in \MC(\fg)$ is a Maurer-Cartan element.

We also recall the Quillen adjunction of rational homotopy theory 
\begin{equation}\label{equ:rht adj}
\Omega \colon \sSet \rightleftarrows \dgca^{op} \colon \G,   
\end{equation}
with $\dgca$ the category of dg commutative algebras, $\Omega=\Mor_{\sSet}(-,\Omega(\Delta^\bullet))$ the PL differential forms functor and $\G=\Mor_{\dgca}(-,\Omega(\Delta^\bullet))$ the geometric realization functor.


\subsection{Proof of Theorem \ref{thm:main flD Q}}
We follow the proof of Theorem \ref{thm:main flD top} above.
We first conduct several preparatory computations.
First, since $B\SO(2)^\Q$ is a $K(\Q,2)$,
\[
\pi_i \Aut_{\Mon}^h(\SO(2)^\Q)
=
\pi_i \Map_*(B\SO(2)^\Q, B\SO(2)^\Q) 
=
\tilde H^{2-i} ((\C P^\infty)^\Q, \Q) 
=
\begin{cases}
\Q & \text{for $i=0$} \\
0 & \text{otherwise}
\end{cases}.
\]
For the last equality we either use Hurewicz' Theorem or the fact that $\C P^\infty$ is $\Q$-good, that is, $H((\C P^\infty)^\Q, \Q)=H(\C P^\infty, \Q)$. We hence conclude that 
$$
\Aut_{\Mon}^h(\SO(2)^\Q)\simeq \Q.
$$

Similarly, We compute that 
\[
    \Map_* (B\SO(2)^\Q, B\Q^\times) \simeq * 
\]
is weakly contractible. For the unbased mapping space it hence follows that 
\begin{equation}\label{equ:map bso bqx}
    \Map(B\SO(2)^\Q, B\Q^\times) \simeq \Q^\times.     
\end{equation}

Next let $\grt_1$ be the (complete) Grothendieck-Teichmüller Lie algebra such that $\GRT_1=\Exp_\bullet(\grt_1)$. The Lie algebra $\grt_1$ has a complete weight-grading, and the pieces $\gr_W \grt_1$ of fixed weight are finite-dimensional.
We define the Lie coalgebra
\[
\grt_1^c:= \bigoplus_{W} (\gr_W \grt_1)^*.
\]
We also define the abelian graded Lie algebra $\fg:=\Q e$, with the generator $e$ concentrated in cohomological degree $-1$.
We then have that $\SO(2)^\Q=\Exp_\bullet(\fg)$.
Finally, for a Lie coalgebra $\fc$ we denote the Chevalley-Eilenberg complex (a dg commutative algebra) by 
\[
C(\fc) =(S(\fc[-1]), D ) .   
\]
Next we compute 
\begin{align*}
\Map(B\SO(2)^\Q, B(\GRT_1\times \SO(2)^\Q))
&=
\Map_{\sSet}(\MC_\bullet(\fg), \MC_\bullet(\grt_1\oplus \fg))
=
\Map_{\sSet}(\MC_\bullet(\fg), \G(C(\grt_1^c\oplus\fg^*)))
\\&=
\Map_{\dgca}(C(\grt_1^c\oplus\fg^*), \Omega(\MC_\bullet(\fg)))
\end{align*}
by adjunction.
Now we can apply \cite[Corollary 1.3]{Berglund} to see that the Chevalley complex of $\fg$ is a Sullivan model for $B\SO(2)^\Q=\MC_\bullet(\fg)$, that is, 
\[
    \Omega(\MC_\bullet(\fg)) \simeq C(\fg^c).  
\]
This means that 
\[
\Map(B\SO(2)^\Q, B(\GRT_1\times \SO(2)^\Q))
\simeq 
\Map_{\dgca}(C(\grt_1^c\oplus \fg^*), C(\fg^*))
=
\MC_\bullet((\grt_1\oplus \fg)\hotimes C(\fg^*))
=
\MC_\bullet((\grt_1\oplus \fg)\hotimes \Q[u]),
\]
where $u$ is a variable of degree $+2$ that is dual to the generator $e$ of $\fg$.
We are interested in the connected component corresponding to the MC element $\alpha=u\otimes x$, corresponding to the map 
\begin{gather*}
    B\SO(2)^\Q\to  B\GRT_1\times B\SO(2)^\Q \\
    x \mapsto * \times x.
\end{gather*}

We then have that 
\[
    \Map(B\SO(2)^\Q, B(\GRT_1\times \SO(2)^\Q))_\alpha
    =
    \MC_\bullet((\grt_1\oplus \fg)\hotimes \Q[u])_\alpha 
    =
    \MC_\bullet(\trunc( ((\grt_1\oplus \fg)\hotimes \Q[u])^\alpha ) )
\]
with $\trunc(-)$ the truncated dg Lie algebra.
But $\alpha$ is in the center of the Lie algebra $(\grt_1\oplus \fg)\hotimes \Q[u]$ and does not produce a differential upon twisting. 
    The truncation is then
    \[
    \trunc(((\grt_1 \oplus \fg)\hotimes \Q[u])^\alpha) 
    =\grt_1 \oplus \fg.
    \]
    Hence 
    \begin{equation}\label{equ:aut ld q proof 1}
        \Map(B\SO(2)^\Q, B(\GRT_1\ltimes \SO(2)^\Q))_\alpha 
        =
        B(\GRT_1\times \SO(2)^\Q).
    \end{equation}

Next consider the fiber sequence of simplicial groups
\[
    \GRT_1\times \SO(2)^\Q
    \to 
    \GRT\ltimes \SO(2)^\Q
    \to 
    \Q^\times,
\]
and an associated fiber sequence 
    \[
    \Map(B\SO(2)^\Q, B(\GRT_1\times \SO(2)^\Q)) 
    \to 
    \Map(B\SO(2)^\Q, B(\GRT\ltimes \SO(2)^\Q)) 
    \to 
    \Map(B\SO(2)^\Q,B\Q^\times).
    \]
    The base is equal to $\Q^\times$ by \eqref{equ:map bso bqx}.
    Restricting to connected components over $1\in \Q^\times$ we hence have 
    \[
        \Map(B\SO(2)^\Q, B(\GRT\ltimes \SO(2)^\Q))_{[1]}
        \simeq 
        \Map(B\SO(2)^\Q, B(\GRT_1\times \SO(2)^\Q)).
    \]
    We take the connected component corresponding to the inclusion of $\SO(2)^\Q$, corresponding to the MC element $\alpha$ above. Using \eqref{equ:aut ld q proof 1} this yields 
    \[
        \Map(B\SO(2)^\Q, B(\GRT\ltimes \SO(2)^\Q))_\alpha 
        \simeq
        B(\GRT_1\times \SO(2)^\Q).
    \]

Finally we use Theorem \ref{thm:main} for the case $\POp=\lD_2^\Q$ and $G=\SO(2)^\Q$ and Theorem \ref{thm:aut lD Q} to obtain from the previous equation that
\[
    \Aut_{\SO(2)^\Q \Op}^h(\lD_2^\Q) \cong \GRT_1\ltimes \SO(2)^\Q.
\]
Furthermore, the homotopy fiber sequence of Theorem \ref{thm:main} then reads,
\begin{equation}\label{equ:long Q}
\GRT_1\ltimes \SO(2)^\Q
\to 
\Aut_{\Op}^h(\flD_2^\Q) 
\to 
\Q^\times.
\end{equation}

The action of $\GRT\ltimes \SO(2)^{\Q}$ on $\lD_2^\Q$ extends to an action on $\flD_2^\Q$. (This action can be explicitly constructed using ribbon braids as in \cite{BHR} or can be deduced from the Lie algebra action on graphical models of $\flD_2^\Q$ as in \cite{Brun}.) Hence the final arrow in \eqref{equ:long Q} induces a surjective map on $\pi_0(-)$ and from the long exact sequence of homotopy groups we readily conclude that 
$$
\Aut_{\Op}^h(\flD_2^\Q) \simeq \GRT\ltimes \SO(2)^{\Q}.
$$
\hfill \qed

\subsection{Proof of Theorem \ref{thm:main flD profinite}}
We generically denote by $X\mapsto \widehat{X}$ the profinite completion functor in the category of groups, groupoids and simplicial sets and by $X\mapsto |X|$ the right adjoint to this functor.

Let $G\mapsto \B G$ be the functor that sends a group to the corresponding groupoid with one object. We let $\pab$ and $\parb$ denote the parenthesized braid operad and parenthesized ribbon braid operad respectively. There is an isomorphism
\[\parb\cong \pab\rtimes \B\Z.\]
(observe that for any abelian group $A$, then $\B A$ is canonically an abelian group object in groupoids).

We let $\hpab$ and $\hparb$ denote their profinie completion. These are operads in profinite groupoids (i.e. pro-objects in the category of groupoids with finitely many morphisms). The functor $|-|$ preserves products and it follows that $|\pab|$ and $|\parb|$ are operads in groupoids. Moreover, we observe that there is an isomorphism of operads in groupoids
\[|\hparb|\cong|\hpab|\rtimes\B|\widehat{\Z}|\]

\begin{prop}
There is a weak equivalence of simplicial monoids
\[\Aut^h_{\Op(\Gpd)}(|\hparb|)\simeq \widehat{\GT}\ltimes B|\widehat{\Z}|\]
\end{prop}

\begin{proof}
First of all, since $|\hparb|$ is an operad in groupoids, we have a weak equivalence
\[\Aut^h_{\Op(\Gpd)}(|\hparb|)\simeq\Aut^h_{\Op}(N|\hparb|)\]
Since the nerve functor preserves product, there is a weak equivalence
\[N|\hparb|\simeq N|\hpab|\rtimes B|\widehat{\Z}|\]
so we can use our main theorem and we get a fiber sequence
\[\Aut^h_{|B\widehat{\Z}|\Op}(N|\hpab|)\to \Aut^h_{\Op}(N|\hparb|)\to\Aut^h_{\Mon}(B|\widehat{\Z}|)\]
The group $\Aut^h_{\Mon}(B|\widehat{\Z}|)$ can be computed as in the previous paragraph and we find
\[\Aut^h_{\Mon}(B|\widehat{\Z}|)\simeq \widehat{\Z}^{\times}\]
On the other hand, the group $\Aut^h_{\Op}(N|\hpab|)$ has been computed in \cite{HorelProfinite} and shown to be equivalent to the semi-direct product $\widehat{\GT}\ltimes B|\widehat{\Z}|$. It follows from the second part of our  main theorem that
\[\Aut^h_{|B\widehat{\Z}|\Op}(N|\hpab|)\simeq \left(\Map(B^2|\widehat{\Z}|,B(\widehat{\GT}\ltimes B|\widehat{\Z}|)\right)_{Bf}\]
We can compute this space using the fiber sequence 
\[B\widehat{\GT}_1\times B^2|\widehat{\Z}|\to B(\widehat{\GT}\ltimes B|\widehat{\Z}|)\to B\widehat{\Z}^\times\]
where $\widehat{\GT}_1$ denotes the kernel of the cylcotomic character
\[\chi:\widehat{\GT}\to\widehat{\Z}^\times.\]
We thus obtain a fiber sequence
\[\Map(B^2|\widehat{\Z}|,B\widehat{\GT}_1\times B^2|\widehat{\Z}|)\to 
\Map(B^2|\widehat{\Z}|,B(\widehat{\GT}\ltimes B|\widehat{\Z}|))\to 
\Map(B^2|\widehat{\Z}|,B\widehat{\Z}^\times)\]
The third space in this fiber sequence is identified with the discrete space $\widehat{\Z}^\times$ while the first space is the product 
\[\Map(B^2|\widehat{\Z}|,B\widehat{\GT}_1)\times \Map(B^2|\widehat{\Z}|,B^2|\widehat{\Z}|)\simeq B\widehat{\GT}_1\times \Map(B^2|\widehat{\Z}|,B^2|\widehat{\Z}|)\]
Since the base in this fiber sequence is discrete, it follows that
\[\left(\Map(B^2|\widehat{\Z}|,B(\widehat{\GT}\ltimes B|\widehat{\Z}|)\right)_{Bf}\simeq B\widehat{\GT}_1\times\left( \Map(B^2|\widehat{\Z}|,B^2|\widehat{\Z}|)\right)_{B^2\id}\simeq B\widehat{\GT}_1\times B^2|\widehat{\Z}|\]
where the last equivalence follows from a computation similar to the one in the previous subsection. Putting everything together, we get a fiber sequence
\[B\widehat{\GT}_1\times B^2|\widehat{\Z}|\to B\Aut^h_{\Op}(N|\hparb|)\to B(|\widehat{\Z}|^\times)\] 
which is what we wanted modulo solving the extension problem. 
The extension problem is solved by noting that the action of $\widehat{\GT}$ on $N|\hparb|$ extends to an action of $\widehat{\GT}\ltimes B|\widehat{\Z}|$.
(Generally, if a group $G$ acts on an operad $\POp$ such that $\POp(1)=H$ is a group, then the action extends to a $G\ltimes H$-action on $\POp$.)
\end{proof}

Let $\widehat{\flD_2}$ denote the profinite completion of the framed little disks operad. This can be viewed as a weak operad (i.e. a functor from the algebraic theory of operads to profinite spaces preserving products up to homotopy) as in \cite{HorelProfinite} or as a Segal dendroidal object in profinite spaces as in \cite{BHR}. In either case, we can compute the homotopy automorphisms in the relevant category and we have the following theorem.

\begin{thm}
We have
\[\Aut^h(\widehat{\flD}_2)\simeq \widehat{\GT}\ltimes B\widehat{\Z}\]
\end{thm}

\begin{proof}
We use the language of weak operads as in \cite{HorelProfinite}. We can argue exactly as in \cite[Corollary 8.12]{HorelProfinite} that there is a weak equivalence
\[\Aut^h(\widehat{\flD}_2)\simeq \Aut^h(|R\widehat{\flD}_2|)\]
where $R$ denotes a fibrant replacement in the model structure of weak operads in profinite spaces. Now, we claim that $N|\widehat{\parb}|$ is weakly equivalent to $|R\widehat{\flD}_2|$ in the model structure of weak operads in profinite spaces. This will conclude the proof thanks to the previous proposition. In order to prove the claim, it suffices to observe that there is a weak equivalence of weak operads in profinite spaces
\[N\hparb\simeq \widehat{\flD}_2\]
which is proved in \cite[Lemma 8.3]{BHR} and that $N\hparb$ is fibrant as a weak operad in profinite spaces. This comes from the fact that the functor $N$ is a right Quillen functor and that the profinite groupoids $\hparb(n)$ are fibrant by \cite[Proposition 4.40]{HorelProfinite}.
\end{proof}

\subsection{A remark about the non-unital case}
Our theorems above are for the unital (framed) little disks operad, i.e. we have $\lD_2(0)\simeq *$. But we also immediately derive an analogous result for the non-unital little disks operad $\lD_2^{nu}$ defined by
\begin{equation}
    \lD_2^{nu}(n) =
    \begin{cases*}
      \varnothing & if $n=0$ \\
      \lD_2(n)       & if $n\geq 1$
    \end{cases*}
  \end{equation}
  with the obvious operad structure ; and also the non-unital framed little disks operad $\flD_2^{nu}$ defined similarly. In that case we have the following theorems.

\begin{thm}\label{thm:non-unital thm}
There are weak equivalences of simplicial monoids
\[\Aut^h_{\SO(2)\Op}(\lD_2^{nu})\simeq \SO(2)\]
\[\Aut^h_{\Op}(\flD_2^{nu})\simeq \O(2)\]
\end{thm}

\begin{proof}
By \cite[Theorem 2.7]{horelremarks}, the map
\[\Aut^h_{\Op}(\lD_2)\to\Aut^h_{\Op}(\lD_2^{nu})\]
is a weak equivalence of simplicial monoids. It follows from Theorem \ref{thm:main} that
\[\Aut^h_{\SO(2)\Op}(\lD_2)\to\Aut^h_{\SO(2)\Op}(\lD_2^{nu})\]
is a weak equivalence as well. We therefore obtain the first claim by Theorem \ref{thm:main flD top}. The second claim is proved exactly as in Theorem \ref{thm:main flD top} using the observation that $\flD_2^{nu}=\lD_2^{nu}\rtimes\SO(2)$.
\end{proof}

We also have an analogue of Theorem \ref{thm:main flD profinite}.

\begin{thm}\label{thm:non-unital thm profinite}
There is a weak equivalence of simplicial monoids
\[\Aut^h(\widehat{\flD_2^{nu}})\simeq \widehat{\GT}\ltimes B\widehat{\Z}\]
\end{thm}

\begin{proof}
By the Quillen adjunction between profinite spaces and spaces, we have identifications
\[\Aut^h(\widehat{\flD_2^{nu}})\simeq \Map^{'h}(\flD_2^{nu},|R\widehat{\flD_2^{nu}}|), \;\;\Aut^h(\widehat{\flD_2})\simeq \Map^{'h}(\flD_2,|R\widehat{\flD_2}|) \]
where $R$ is a fibrant replacement in the category of weak operads in profinite spaces. By \cite[Theorem 2.7]{horelremarks}, we have a weak equivalence
\[\Map^h(\flD_2,|R\widehat{\flD_2}|)\simeq \Map^h(\flD_2^{nu},|R\widehat{\flD_2^{nu}}|)\]
Putting everything together, we find that the map
\[\Aut^h(\widehat{\flD_2})\to \Aut^h(\widehat{\flD_2^{nu}})\]
is a weak equivalence.
\end{proof}

\begin{rem}
We believe that Theorem \ref{thm:main flD Q} also holds for the non-unital framed little disks operad, however, this does not follow immediately from \cite[Theorem 2.3]{horelremarks} since rationalization is not a localization.
\end{rem}

\section{Homotopy automorphisms of the Batalin-Vilkovisky cooperad}

It has been shown by B. Fresse \cite{Frextended} that the classical Quillen adjunction \eqref{equ:rht adj} of rational homotopy theory can be extended to a Quillen adjunction 
\[
\Omega_\sharp\colon  \Op_* \rightleftarrows (\La\HOpc)^{op}
\colon \G     
\]
between the category of simplicial ($\Lambda$) operads with the Reedy model structure and the category of dg $\La$ Hopf cooperads $\La\HOpc$, that is, $\La$ cooperads in the category of dg commutative algebras.
In this setting, the rationalization of a cofibrant operad $\POp\in \Op_*$ is defined as 
\[
\POp^\Q := \G (\widehat{ \Omega_\sharp\POp} ),     
\]
where $\widehat {\Omega_\sharp\POp}$ is a fibrant replacement of $\Omega_\sharp\POp$ in $\La\HOpc$.

By formality of the framed little disks operad \cite{Severa2010, GiansiracusaSalvatore}, we know that $\Omega_\sharp(\flD_2)\simeq H^\bullet(\flD_2;\Q):=\BV^c$, with $\BV^c$ the Batalin-Vilkovisky cooperad.
By adjunction we have that 
\[
 \Aut_{\La\HOpc}^h(\BV^c)
\cong 
\Map_{\Op_*}^{'h}(\flD_2, \flD_2^\Q)
\simeq \Map_{\Op}^{'h}(\flD_2, \flD_2^\Q).
\]
This is a priori different from the simplicial monoid computed in Theorem \ref{thm:main flD Q}.
(The underlying problem is that $\flD_2$ is not known to be $\Q$-good.)
However, we may compute $\Aut_{\La\HOpc}^h(\BV^c)$ along the same lines as above to yield the following result.

\begin{thm}\label{thm:BV}
    We have a weak equivalence of simplicial monoids
    \[
        \Aut_{\La\HOpc}^h(\BV^c) \simeq  
        \Aut_{\Op}^h(\flD_2^\Q)
        \simeq 
        \GRT\ltimes \SO(2)^\Q.
    \]
\end{thm}
\begin{proof}[Proof sketch.]
We may adapt the proof of Theorem \ref{thm:main} above for the case of $\La$ Hopf cooperads instead of topological operads.
One has a homotopy fiber sequence 
\[
    \Aut^h_{(\La\HOpc)^{/A}}(\BV^c) 
    \to 
    \Aut_{\La\HOpc}^h(\BV^c) 
    \to 
    \Aut_{\dgca}(A)
\]
with $A:=\BV^c(1)=H^\bullet(S^1)$.
We furthermore have, by the Hopf cooperad analog of the adjunction of Lemma \ref{lem:semidirect adj},
\[
    \Aut^h_{(\La\HOpc)^{/A}}(\BV^c) 
    \simeq 
    \Map^{'h}_{\SO(2)^\Q\La\HOpc}(\BV^c, \e_2^c) 
\]
with $\e_2^c=H^\bullet(\lD_2)$ the Gerstenhaber cooperad.
Lemma \ref{lem:QG ComG fib} and Corollary \ref{cor:only one} also have Hopf cooperad analogs that together yield 
\[
    \Map^{'h}_{\SO(2)^\Q\La\HOpc}(\BV^c, \e_2^c) 
    \simeq \Aut^h_{\SO(2)^\Q\La\HOpc}(\e_2^c),
\]
so that the inclusion
\[
    \Aut^h_{\SO(2)^\Q\La\HOpc}(\e_2^c) \to \Aut^h_{(\La\HOpc)^{/A}}(\BV^c) 
\]
is a weak equivalence of simplicial monoids.
By the computations of \cite{Frbook, FWAut} we have that 
\[
    \Aut^h_{\La\HOpc}(\e_2^c) \cong \GRT \ltimes \SO(2)^\Q.    
\]
Then Proposition \ref{prop : equivariant Aut} tells us that 
\[
    B\Aut^h_{\SO(2)^\Q\La\HOpc}(\e_2^c)  
    \simeq 
    \Map(B\SO(2)^\Q, B\Aut^h_{\La\HOpc}(\e_2^c)).
\]
From this point on the computations are the same as in the proof of Theorem \ref{thm:main flD Q} above, and yield the stated result.
\end{proof}




\bibliographystyle{amsalpha}

\end{document}